\documentclass[11pt]{amsart}

\usepackage{amsmath}
\usepackage{amssymb}
\usepackage{bbm}
\usepackage{pdfsync}
\usepackage{cleveref}
\usepackage{mathtools}
\usepackage{xcolor}
\usepackage{xparse}

\date{\today}
\allowdisplaybreaks
\newcommand{\R}{{\mathbb R}}       
\newcommand{\N}{{\mathbb N}}
\newcommand{\T}{{\mathbb T}}
\newcommand{\Sp}{{\mathbb S}^1}
\newcommand{\Z}{{\mathbb Z}}       

\DeclarePairedDelimiter\abs{\lvert}{\rvert}
\DeclarePairedDelimiter\norm{\lVert}{\rVert}

\newcommand{\vertiii}[1]{{\left\vert\kern-0.25ex\left\vert\kern-0.25ex\left\vert #1 
		\right\vert\kern-0.25ex\right\vert\kern-0.25ex\right\vert}}

\DeclarePairedDelimiterX{\set}[1]{\{}{\}}{\setargs{#1}}
\NewDocumentCommand{\setargs}{>{\SplitArgument{1}{;}}m}
{\setargsaux#1}
\NewDocumentCommand{\setargsaux}{mm}
{\IfNoValueTF{#2}{#1} {#1\,\delimsize|\,\mathopen{}#2}}

\newtheorem{theorem}{Theorem}[section]
\newtheorem{lemma}[theorem]{Lemma}

\newtheorem*{theorem*}{Theorem}

\theoremstyle{definition}
\newtheorem{definition}[theorem]{Definition}

\numberwithin{equation}{section}

\begin{document}
	
	\title[On Sharpness of $L\log L$ Criterion]{On Sharpness of $L\log L$ Criterion for Weak Type $(1,1)$ boundedness of rough operators}

	\author{Ankit Bhojak}
	\address{Ankit Bhojak\\
		Department of Mathematics\\
		Indian Institute of Science Education and Research Bhopal\\
		Bhopal-462066, India.}
	\email{ankitb@iiserb.ac.in}
	\thanks{}

	\begin{abstract}
	In this note, we show that the $L\log L$ hypothesis is the strongest size condition on a homogeneous rough function on the sphere which ensures the weak type $(1,1)$ boundedness of the corresponding singular integral $T_\Omega$, provided $T_\Omega$ is bounded  in $L^2$.
	\end{abstract}

\subjclass[2010]{Primary 42B20}	
	\maketitle
	
	\section{Introduction}
	Let   $\Omega\in L^1(\mathbb S^{d-1})$ with $\int_{\mathbb S^{d-1}}\Omega(\theta) d\theta=0$, where $d\theta$ is the surface measure on $\mathbb S^{d-1}$. Calder\'on and Zygmund \cite{CZ} considered the rough singular integrals defined as,
	\[ T_{\Omega}f(x)=p.v.\int\frac{1}{|x-y|^d}\Omega\Bigl(\frac{x-y}{|x-y|}\Big)f(y)\;dy,\]
	
	They showed that  $\Omega\in L \log L(\mathbb S^{d-1})$  i.e. $\int_{\Sp}|\Omega(\theta)|\log(e+|\Omega(\theta)|)<\infty$ implies that $T_\Omega$ is bounded on $L^p(\R^d)$ for $1<p<\infty$. The singular integral $T_\Omega$ was shown to be of weak type $(1,1)$ using $TT^*$ arguments by Christ and Rubio de Francia \cite{CR} in dimension $d=2$ (and independently by Hofmann \cite{Hof}). The case of general dimensions was resolved by Seeger \cite{S1} by showing that $T_\Omega$ is of weak type $(1,1)$ for $\Omega\in L\log L(\mathbb S^{d-1})$ assuming the $L^2$ boundedness of $T_\Omega$.
	
	\par
	
	It is of interest to know other sufficient conditions on $\Omega$ that ensures the weak type boundedness of the operator $T_\Omega$. In fact, during the inception of this problem, Calder\'on and Zygmund \cite{CZ} showed that $\Omega\in L\log L$ is ``almost" a necessary size condition for $T_\Omega$ to be $L^2$ bounded. If we drop the condition that $\Omega\in L\log L$, then Calder\'on and Zygmund \cite{CZ} pointed out that $T_\Omega$ may even fail to be $L^2$ bounded. Infact, the examples of $\Omega$ constructed in \cite{WZ} lies outside the space $L\log L$ and the corresponding operator $T_\Omega$ is unbounded on $L^2(\R^d)$. Later on, it was shown in \cite{Co,RW} that $\Omega\in H^1(\Sp)$ in the sense of Coifman and Weiss \cite{CW} implies $T_\Omega:L^p(\R^d)\to L^p(\R^d),\;1<p<\infty$. It is still an open problem if $T_\Omega$ is of weak type $(1,1)$ for $\Omega\in H^1(\Sp)$. A partial result assuming additional conditions on $H^1$-atoms in dimension two was obtained by Stefanov \cite{Sv}.
	
	In \cite{GHR, H2}, it was shown that $T_\Omega$ distinguishes $L^p$ spaces by considering a suitable quantity based on the Fourier transform of $\Omega$. However, we would like to know if there exists an Orlicz space $X\supsetneq L\log L$ which would ensure that the $L^2$ boundedness of $T_\Omega$ implies the weak $(1,1)$ boundeness of $T_\Omega$ when $\Omega\in X$. We will show that no such $X$ exists. To state our main result, we introduce the Orlicz spaces and discuss some of its basic properties.
	\par
	\begin{definition}[\cite{BS}]
		Let $\Phi:[0,\infty)\to [0,\infty)$ be a Young's function i.e. there exists an increasing and left continuous function $\phi:[0,\infty)\to [0,\infty)$ with $\phi(0)=0$ such that $\Phi(t)=\int_0^t\phi(u)\;du$ and $\frac{\Phi(t)}{t}\to\infty,\text{ as }t\to\infty$. We say $\Omega\in \Phi(L)(\Sp)$, if the quantity
		\begin{equation}\label{O1}
			\|\Omega\|_{\Phi(L)}=\int_{\mathbb{S}^1}\Phi(|\Omega(\theta)|)\;d\theta
		\end{equation}
		is finite.	
	\end{definition}
	The quantity in \eqref{O1} fails to be a norm and $\Phi(L)(\Sp)$ is not even a linear space. To remedy that, we define the set
	\[L^\Phi(\Sp)=\{\Omega:\Sp\to\R: \exists k>0 \text{ such that }\norm{k^{-1}\Omega}_{\Phi(L)}<\infty\}.\]
	We define the Luxemburg norm as
	\[\vertiii{\Omega}_{\Phi(L)}=\inf\{k>0: \|k^{-1}\Omega\|_{\Phi(L)}<\infty\}.\]
	It is well known that the Orlicz space $L^\Phi(\Sp)$ forms a Banach space with this norm. We state the following fact that compares this norm with the quantity in \eqref{O1}, see Lemma $8.8$ in \cite{BS}.
	\begin{lemma}\label{Orlicz}
		If $\|\Omega\|_{\Phi(L)}<\vertiii{\Omega}_{\Phi(L)}$, then $\vertiii{\Omega}_{\Phi(L)}\leq 1$.
	\end{lemma}
	\section{Main result}
	We state our main result for dimension two but the same also holds for higher dimensions using the methods in \cite{WZ,GHR}. Our main result is the following,
	\begin{theorem}\label{logoptimal}
		Let $\Phi$ be a Young's function such that
		\begin{align}
			\Psi(t)=\frac{t\log(e+t)}{\Phi(t)}&\to\infty,\;\text{as}\;t\to\infty,
		\end{align}
	Then there exists an $\Omega\in \Phi(L)(\mathbb S^1)$ such that $T_\Omega$ is $L^p$ bounded iff $p=2$. In particular, $T_\Omega$ does not map $L^1(\R^2)$ to $L^{1,\infty}(\R^2)$.
	\end{theorem}
	We note that using the geometric construction in \cite{H2}, one can obtain the above theorem for the space $L(\log L)^{1-\epsilon}(\Sp),\;0<\epsilon\leq 1$. To obtain the general case, we will employ the construction in \cite{GHR} with a suitable modification to ensure that the resulting $\Omega$ lies in the required Orlicz space.
	\section{Proof of \Cref{logoptimal}}
	To prove \Cref{logoptimal}, we will rely on a transference principle \cite{GHR}. The space of $L^p$ multipliers $M^p(\T)$ is defined as
	\[M^p(\T)=\{\textbf{a}=\{a_n\}\in l^\infty(\Z):\; T_{\textbf{a}}f(x)=\sum_{n\in\Z}a_n\widehat f(n)e^{2\pi inx} \text{ is bounded on } L^p(\T)\},\]
	with $\norm{\textbf{a}}_{M^p(\T)}=\norm{T_{\textbf{a}}}_{L^p(\T)\to L^p(\T)}$. 
	
	The idea is to construct a sequence of $\{\Omega_n\}$ such that the $L^p$ norm of $T_{\Omega_n}$ is large for $p\neq 2$. We achieve this by employing the fact that $\{e^{2\pi ikx},\;k\in\Z\}$ is not an unconditional basis for $L^p([0,1]),\;p\neq 2$. We have,
	\begin{lemma}[\cite{GHR}]\label{Riesz}
		For $p\neq 2$ and fixed $n\in\N$, there exists finite sequences $\{a_k\}_{k=1}^{n}$ and $\{\epsilon_k\}_{k=1}^{n}$ (depending on $n$) with $\epsilon_k\in\{-1,1\}$ such that
		\[\norm[\Big]{\sum\limits_{k=1}^n\epsilon_k a_k e^{2\pi ikx}}_{L^p(\mathbb [0,1])}\geq c_p n^{\abs{\frac{1}{2}-\frac{1}{p}}}\norm[\Big]{\sum\limits_{k=1}^n a_k e^{2\pi ikx}}_{L^p(\mathbb [0,1])},\]
		where $c_p>0$ depends only on $p$. Consequently, $\norm{\{\dots,0,\epsilon_1,\epsilon_2,\dots,\epsilon_n,0,\dots\}}_{M^p(\T)}\gtrsim n^{\abs{\frac{1}{2}-\frac{1}{p}}}$. Moreever, we can choose $\epsilon_k$ such that
		\[\norm{\{\dots,0,\epsilon_1,\epsilon_2,\dots,\epsilon_n,0,\dots\}}_{M^p(\T)}=\sup\{\norm{\{\dots,0,\delta_1,\delta_2,\dots,\delta_n,0,\dots\}}_{M^p(\T)}:\;|\delta_k|\leq 1\}.\]
	\end{lemma}
	\begin{proof}
		The inquality follows from the unconditionality of basis $\{e^{2\pi ikx},\;k\in\Z\}$ for $p\neq 2$ with constant $K(n)\to\infty$ as $n\to\infty$. We justify the choice of constant $c_pn^{\abs{\frac{1}{2}-\frac{1}{p}}}$. Indeed, we invoke Theorem 1 from \cite{R}. 
		
		\emph{For $n\in\N$, there exists $\{\epsilon_k\}_{k=1}^n$ with $\epsilon_k=\pm1$ such that $\norm{\sum_{k=1}^n\epsilon_k e^{2\pi ikx}}_{L^\infty([0,1])}\leq 5n^{\frac{1}{2}}.$}\\
		
		By using the well known fact that $L^p$ norm of the Dirichlet kernel satisfies the following estimate:
		\[\norm{\sum_{k=1}^n\ e^{2\pi ikx}}_{L^p([0,1])}\sim n^{1-\frac{1}{p}}.\]
		we have, for $p>2$ and $a_k=\epsilon_k$,
		\[\norm[\Big]{\sum\limits_{k=1}^n\epsilon_k^2 e^{2\pi ikx}}_{L^p(\mathbb [0,1])}\geq c_p n^{\frac{1}{2}-\frac{1}{p}}\norm[\Big]{\sum\limits_{k=1}^n \epsilon_k e^{2\pi ikx}}_{L^p(\mathbb [0,1])}.\]
		The analogous inequality for $p<2$ follows from duality.
	\end{proof}
	If we have a seqeunce of multiplier $\{\gamma_n\}$ on $\R^2$ such that $\gamma_n\vert_{\Z}=\{\dots,0,\epsilon_1,\epsilon_2,\dots,\epsilon_n,0,\dots\}$ where $\epsilon_k$ are from the above lemma, then by classical transference $\norm{T_{\gamma_n}}_{L^p(\R^2)\to L^p(\R^2)}\gtrsim n^{\abs{\frac{1}{2}-\frac{1}{p}}}$. Our aim is to produce $\gamma_n$ such that $T_{\gamma_n}=T_{\Omega_n}$. Towards this, the following observation is important.
	\begin{lemma}[\cite{GHR}]\label{Transference}
		Let $1<p<\infty$ and $\gamma\in M^p(\R^2)$ be continuous on an arithmetic progression $\{x_k\}_{k=1}^n$ in $\R^2$ (i.e. there exists vector $v\in\R^2$ such that $x_k-x_{k-1}=v$). Then there exists a constant $C_p>0$ such that
		\[\norm{\gamma}_{M^p(\R^2)}\geq C_p \norm{\{\dots,0,\gamma(x_1),\gamma(x_2),\dots,\gamma(x_n),0,\dots\}}_{M^p(\T)}.\]
	\end{lemma}
	The proof of above lemma can be found in \cite{GHR}. We now begin the proof of \Cref{logoptimal}.
	
	\begin{proof}[\textbf{Proof of \Cref{logoptimal}}]
		We fix a large $N\in\N$. Let $n=\left[\Psi\left(\frac{N}{\log N}\right)\right]$. By hypothesis, we have $n\to\infty$ as $N\to\infty$.
		
		Let $s_n\in\N$ be a large number such that there exists natural numbers $t_1,t_2,\dots,t_{2n}$ satisfying
		\begin{itemize}
			\item The numbers $t_k$ are in arithmetic progression, i.e. $t_{k+1}-t_{k}=t_{k}-t_{k-1}$.
			\item Let $x_k=(t_k,s_n)\in\R^2$. Then $x_k,\; k=1,\dots,2n$ lies in the second quadrant between the lines $y$-axis and $y=-x$.
			\item $\abs{\frac{x_{k+1}}{|x_{k+1}|}-\frac{x_k}{|x_k|}}\sim \frac{1}{n}$.
		\end{itemize}
		
		We denote $\tilde{x}_k$ to be the point on $\mathbb S^1$ obtained by rotating the point $\frac{x_k}{|x_k|}$ by $\frac{\pi}{2}$ radians clockwise. We consider $I_k,\;k=1,\dots,2n$ to be the arc on $\mathbb S^1$ with centre $\tilde x_k$ and arc length $N^{-1}$ and denote $\mathfrak R_{\alpha}(I_k)$ to be the arc obtained by rotating $I_k$ by $\alpha$ radians counterclockwise. 
		
		We define $w_k$ as
		\[w_k(\theta)=c_{I_k}(-\chi_{_{I_k}}(\theta)+\chi_{_{\mathfrak R_{\frac{\pi}{2}}(I_k)}}(\theta)-\chi_{_{\mathfrak R_{\pi}(I_k)}}(\theta)+\chi_{_{\mathfrak R_{\frac{3\pi}{2}}(I_k)}}(\theta)),\]
		where we choose $c_{I_k}$ as follows. We recall that the Fourier transform of the kernel in $T_\Omega$ for any even $\Omega$ with mean value zero is given by
		\[\widehat K_\Omega(\xi)=\int_{\Sp}\Omega(\theta)\log\frac{1}{|\langle\xi,\theta\rangle|}\;d\theta.\]
		We define the larger quantity $m(\Omega)$ which will be useful for our purpose.
		\[m(\Omega)(\xi):=\int_{\Sp}|\Omega(\theta)|\log\frac{1}{|\langle\xi,\theta\rangle|}\;d\theta.\]
		Clearly, $|\widehat{K}_\Omega(\xi)|\leq m(\Omega)(\xi)$. 
		
		We choose $c_{I_k}$ such that $m(w_k)(\frac{x_k}{|x_k|})=1$. It is not difficult to see that $c_{I_k}$ and $\widehat K_{w_k}(\frac{x_k}{|x_k|})$ are independent of $k$ and satisfies,
		\begin{align*}
			c_{I_k}&\sim \frac{N}{\log N}.\\
			1\lesssim\sup_x|\widehat K_{w_k}(x)|=\Bigl|\widehat K_{w_k}&\Bigl(\frac{x_k}{|x_k|}\Bigr)\Bigr|\leq\sup_x m(w_k)(x)=1.
		\end{align*}
		We now set 
		\begin{equation*}
			\Omega_n=\sum_{k=1}^{2n}(-1)^k\epsilon_{\left[\frac{k+1}{2}\right]}w_k,
		\end{equation*}
		where $[\;]$ denotes the integer part and $\epsilon_{[.]}$ is as in \Cref{Riesz}.
		
		By the disjointness of the arcs $I_k$, we have 
		\begin{align*}
			\norm{\Omega_n}_{\Phi(L)(\mathbb S^1)}&\sim \sum_{k=1}^{2n}\int_{I_k}\Phi(c_{I_k})\\
			&\sim n N^{-1}\Phi\left(\frac{N}{\log N}\right)\\
			&\sim 1.
		\end{align*}
		Moreover, since $\norm{\Omega_n}_{\Phi(L)(\mathbb S^1)}>1$, by \Cref{Orlicz}, we have $\vertiii{\Omega_n}_{\Phi(L)(\mathbb S^1)}\lesssim 1$.
		
		To calculate the $L^2$ norm of $T_{\Omega_n}$, we will need the following estimate. Let $J_k$ be the arc of length $\frac{1}{100n}$ and whose bisector passes though the point $\frac{x_k}{|x_k|}$. Then for $x\in\mathbb S^1$ lying in second quadrant between the lines $y$-axis and $y=-x$ with $x\notin\bigcup_{i=0}^3 \mathfrak{R}_{\frac{i\pi}{2}}(J_k)$, we have
		\begin{equation}
			m(w_k)(x)\lesssim \frac{\log n}{\log N}.\label{formula7}
		\end{equation}
		The equation \eqref{formula7} follows from the fact that for $\gamma\in (2I_k)^c\cap(0,\frac{\pi}{4})$, we have
		\begin{equation*}
			|m(w_{I_k})(e^{i\gamma})|\lesssim\frac{|\log|\gamma-\tilde x_k||}{|\log |I_k||}.
		\end{equation*}
		Indeed, for $\theta\in I_k$, we have $|\gamma-\tilde x_{k}|<|\theta-\tilde x_{k}|+|\theta-\gamma|<\frac{|I_k|}{2}+|\theta-\gamma|<|\gamma-\tilde x_{k}|/2+|\theta-\gamma|$. Thus $\frac{|\tilde x_{k}-\gamma|}{2}<|\theta-\gamma|$ and it follows that
		\begin{align*}
			|m(w_{I_k})(e^{i\gamma})|&\lesssim-c_{I_k}\int_{I_k}\log|\sin(\theta-\gamma)|\;d\theta\\
			&\leq c_{I_k}|I_k||\log\left|\sin\left(\frac{|\gamma-\tilde x_k|}{2}\right)\right|\\
			&\lesssim \frac{|\log|\gamma-\tilde x_k||}{|\log |I_k||}.
		\end{align*}
		Since $\frac{\Phi(t)}{t}\to\infty$ as $t\to\infty$, we have $\frac{\Psi(t)}{\log t}\to 0$ as $t\to\infty$. Therefore,
		\begin{equation}\label{formula9}
			\norm{m(\Omega_n)}_{L^\infty(\mathbb S^1)}\lesssim 1+\frac{n\log n}{\log N}\lesssim \frac{\Psi\left(\frac{N}{\log N}\right)}{\log N}\log \Psi\left(\frac{N}{\log N}\right)\lesssim \log n.
		\end{equation}
		
		Now we turn to the estimate of $L^p$ bounds of $T_{\Omega_n}$, We claim
		\begin{equation}\label{formula10}
			\norm{T_{\Omega_n}}_{L^p(\R^2)\to L^p(\R^2)}\gtrsim n^{\abs{\frac{1}{2}-\frac{1}{p}}}.
		\end{equation}
		To acheive the above claim, we need to prove that for $1\leq k\leq n$ and $x\in\mathbb S^1$ lying in second quadrant between the lines $y$-axis and $y=-x$ with $x\notin \bigl(\bigcup_{i=0}^3 \mathfrak{R}_{\frac{i\pi}{2}}(J_{2k})\bigr)\cup\bigl(\bigcup_{i=0}^3 \mathfrak{R}_{\frac{i\pi}{2}}(J_{2k-1})\bigr)$, we have
		\begin{equation}\label{formula8}
			|\widehat K_{w_{2k}}(x)-\widehat K_{w_{2k-1}}(x)|\lesssim \left(n\log N\left|\frac{x}{|x|}-\frac{x_{2k}}{|x_{2k}|}\right|\right)^{-1}.
		\end{equation}
		Indeed, let $e^{i\theta_{2k}}=\frac{x_{2k}}{|x_{2k}|},\; e^{i\gamma}=\frac{x}{|x|}$ and $A_{2k}$ be the interval in $(-\frac{\pi}{4},\frac{\pi}{4})$ such that $I_{2k}-\frac{x_{2k}}{|x_{2k}|}=\{e^{i\theta}:\; \theta\in A_{2k}\}$. By using mean value theorem twice and the fact that $|\theta_{2k}-\theta_{2k-1}|$ is small, we have
		\begin{align*}
			|\widehat K_{w_{2k}}(x)-\widehat K_{w_{2k-1}}(x)|&\lesssim c_{I_{2k}}\int_{A_{2k}}\left(\log \frac{1}{|\tan (\theta+\theta_{2k}-\gamma)|}-\log\frac{1}{|\tan (\theta+\theta_{2k-1}-\gamma)|}\right)\;d\theta\\
			&\lesssim c_{I_{2k}}\int_{A_{2k}}\frac{|\tan(\theta+\theta_{2k}-\gamma)-\tan(\theta+\theta_{2k-1}-\gamma)|}{|\tan(\theta+\theta_{2k}-\gamma)|}\;d\theta\\
			&\lesssim c_{I_{2k}}\int_{A_{2k}} \frac{|\theta_{2k}-\theta_{2k-1}|}{|\theta+\theta_{2k}-\gamma|}\;d\theta\\
			&\lesssim \frac{c_{I_{2k}}}{n}\int_{A_{2k}}\frac{1}{|\gamma-\theta_{2k}|}\;d\theta\\
			&\lesssim\left(n\log N\left|\frac{x}{|x|}-\frac{x_{2k}}{|x_{2k}|}\right|\right)^{-1},
		\end{align*}
		where we have used $|\gamma-\theta_{2k}|\leq 2|\theta+\theta_{2k}-\gamma|$ and $\tan \theta\sim\theta$ away from odd multiples of $\frac{\pi}{2}$.\\
		Now, we return to the proof of the claim \eqref{formula10}. For $1\leq k\leq n$, we have
		\[\widehat K_{\Omega_n}(x_{2k})=(-1)^{2k}\widehat K_{w_{2k}}(x_{2k})\epsilon_k+\sum\limits_{1\leq i\neq 2k\leq 2n}(-1)^i\epsilon_{\left[\frac{i+1}{2}\right]}\widehat K_{w_i}(x_{2k})=D\epsilon_k+\delta_k,\]
		where $D=\widehat K_{w_{2k}}(x_{2k})$ and $\delta_k=\sum\limits_{1\leq i\neq 2k\leq 2n}(-1)^i\epsilon_{\left[\frac{i+1}{2}\right]}\widehat K_{w_i}(x_{2k})$.
		
		Using \eqref{formula7} for the term $i=2k-1$ and \eqref{formula8} for the remaining terms (in pair), we get 
		\[|\delta_k|\leq C\left(\frac{\log n}{\log N}+\frac{1}{\log N}\sum_{i=1}^{2n}\frac{1}{i}\right)\leq \frac{C'\log n}{\log N}\leq \frac{|D|}{4} \text{ (for large }n).\] Hence by choice of \Cref{Riesz}, we have
		\[\frac{1}{2}\norm{\{\dots,0,\epsilon_1,\epsilon_2,\dots,\epsilon_{n},0,\dots\}}_{M^p(\T)}\geq\norm[\Big]{\Bigl\{\dots,0,\frac{\delta_1}{D},\frac{\delta_2}{D},\dots,\frac{\delta_{n}}{D},0,\dots\Bigr\}}_{M^p(\T)}.\]
		Since $\widehat K_{\Omega_n}(\theta)$ is a circular convolution of a $L^1(\mathbb S^1)$ and $L^\infty(\mathbb S^1)$, it is continuous at the points $x_{2k},k=1,\dots,n$ and applying \Cref{Transference}, we have
		\begin{align*}
			&\norm{T_{\Omega_n}}_{L^p(\R^2)\to L^p(\R^2)}\\
			=&\norm{\widehat K_{\Omega_n}}_{M^p(\R^2)}\\
			\gtrsim& \norm{\{\dots,0,\widehat K_{\Omega_n}(x_2),\widehat K_{\Omega_n}(x_4),\dots,\widehat K_{\Omega_n}(x_{2n}),0,\dots\}}_{M^p(\T)}\\
			\gtrsim&|D|\left(\norm{\{\dots,0,\epsilon_1,\epsilon_2,\dots,\epsilon_{n},0,\dots\}}_{M^p(\T)}-\norm[\Big]{\Bigl\{\dots,0,\frac{\delta_1}{D},\frac{\delta_2}{D},\dots,\frac{\delta_{n}}{D},0,\dots\Bigr\}}_{M^p(\T)}\right)\\
			\geq& \frac{|D|}{2} \norm{\{\dots,0,\epsilon_1,\epsilon_2,\dots,\epsilon_{n},0,\dots\}}_{M^p(\T)}\\
			\gtrsim& n^{\abs{\frac{1}{2}-\frac{1}{p}}},
		\end{align*}
		where we used \Cref{Riesz} in the last step.
		We conclude the proof by an application of uniform boundedness principle. Indeed, We define the space,
		\[\mathfrak{B}:=\{\Omega:\mathbb S^1\to\R \text{ is even}:\int\Omega=0 \text{ and } \|\Omega\|_{\mathfrak{B'}}=\vertiii{\Omega}_{\Phi(L)(\mathbb S^1)}+\|m(\Omega)\|_{L^\infty(\mathbb S^1)}<\infty\}.\]
		The space $\mathfrak{B}$ forms a Banach space.
		
		Fix $p\neq 2$. For $\mathfrak{F}=\{f\in L^p(\R^2): \|f\|_p=1\}$, we define a collection of operators $\varTheta_f:\mathfrak{B}\to L^p$ as $\varTheta_f(\Omega)=T_\Omega(f)$. Suppose we have
		\[\|T_\Omega\|_{L^p(\R^2)\to L^p(\R^2)}=\sup\limits_{f\in\mathfrak{F'}}\|T_\Omega f\|_{L^p(\R^2)}<\infty,\;\forall\Omega\in\mathfrak{B}.\]
		Then by uniform boundedness principle, there exists $M>0$ such that
		\[\|T_\Omega\|_{L^p(\R^2)\to L^p(\R^2)}=\sup\limits_{f\in\mathfrak{F}}\|\varTheta_f(\Omega)\|_{L^p(\R^2)}<M\|\Omega\|_\mathfrak{B}.\]
		which along with \eqref{formula9} and \eqref{formula10} implies that
		\begin{align*}
			n^{\abs{\frac{1}{2}-\frac{1}{p}}}&\lesssim\norm{T_{\Omega_n}}_{L^p(\R^2)\to L^p(\R^2)}\\
			&\lesssim\vertiii{\Omega_n}_{\Phi(L)(\mathbb S^1)}+\|m(\Omega_n)\|_{L^\infty(\mathbb S^1)}\\
			&\lesssim \log n.
		\end{align*}
		which is a contradiction for large $n$ and $p\neq 2$ and that concludes the proof of \Cref{logoptimal}.
	\end{proof}
	\section*{Acknowledgement}
	I would like to thank Prof. Parasar Mohanty and Prof. Adimurthi for various useful discussions regarding the problem. I acknowledge the financial support from Science and Engineering Research Board, Department of Science and Technology, Govt. of India, under the scheme Core Research Grant, file no. CRG/2021/000230.
		
\end{document}